\documentclass[11pt]{amsart}
\usepackage{amsmath,amsfonts,amsthm,amscd,amssymb,graphicx}
\numberwithin{equation}{section}

\newtheorem{theorem}{Theorem}[section]

\newtheorem{lemma}[theorem]{Lemma}

\newtheorem{proposition}[theorem]{Proposition}
\newtheorem{prop}[theorem]{Proposition}

\newtheorem{cor}[theorem]{Corollary}

\newtheorem{counterexample}[theorem]{Counterexample}
\newtheorem{remark}[theorem]{Remark}
\newtheorem{definition}[theorem]{Definition}

\newcommand{\RR}{\mathbb{R}}
\newcommand{\cO}{\mathcal{O}}
\newcommand{\cS}{\mathcal{S}}

\newcommand{\Id}{{\rm Id }}

\def\bb1{{1\!\!1}}

\def\cG{\mathcal{G}}

\def\cR{\mathcal{R}}

\def\bU{{\bar{U}}}

\def\tx{\tilde x}

\def\R{\Re e}
\def\I{\Im m}

\def\diag{\mbox{diag}}
\newcommand\ul[1]{\underline{#1}}

\newcommand{\wprod}[1]{\langle{#1}\rangle}

\newcommand{\txi}{{\tilde \xi}}

\def\cg{{\hat{\gamma}}}

\begin{document}

\title[Multi-d viscous 
boundary layers] {On asymptotic stability of noncharacteristic
viscous 
boundary layers}

\author[Toan Nguyen]{Toan Nguyen}

\date{\today}

\thanks{I would like to thank Professor Kevin Zumbrun for his many advices, support,
and helpful discussions. This work was supported in part by the
National Science Foundation award number DMS-0300487.}

\address{Department of Mathematics, Indiana University, Bloomington, IN 47402}
\email{nguyentt@indiana.edu}

\maketitle

\begin{abstract} We extend our recent work with K. Zumbrun 
on long-time stability of multi-dimensional noncharacteristic viscous
boundary layers of a class of symmetrizable hyperbolic-parabolic
systems. Our main improvements are (i) to establish the stability
for a larger class of systems in dimensions $d\ge 2$, yielding the
result for certain magnetohydrodynamics (MHD) layers; (ii) to drop
a technical assumption 
on the so--called glancing set which
was required in 
previous works. We also provide a different proof of low-frequency
estimates by employing the method of Kreiss' symmetrizers, replacing
the one relying on detailed derivation of pointwise bounds on the
resolvent kernel.
\end{abstract}

\tableofcontents



\section{Introduction}
Boundary layers occur in many physical settings, such as gas
dynamics and magnetohydrodynamics (MHD) with inflow or outflow
boundary conditions, for example the flow around an airfoil with
micro-suction or blowing. Layers satisfying such boundary conditions
 are called noncharacteristic layers; see, for example, the physical discussion in \cite{S,SGKO}. See also
\cite{GMWZ5,YZ,NZ1,NZ2,Z5} for further discussion.

In this paper, we study the stability of boundary layers assuming
that the layer is noncharacteristic. Specifically, we consider a
boundary layer, or stationary solution, connecting the endstate
$U_+$:
\begin{equation}\label{profile}
\tilde U=\bU(x_1), \quad \lim_{x_1\to +\infty} \bU(x_1)=U_+.
\end{equation}
of a general system of viscous conservation laws on the
quarter-space
\begin{equation}\label{sys}
\tilde U_t +  \sum_jF^j(\tilde U)_{x_j} = \sum_{jk}(B^{jk}(\tilde
U)\tilde U_{x_k})_{x_j}, \quad x\in \mathbb{R}_+^{d},\quad t>0,
\end{equation}
$\tilde U,F^j\in \mathbb{R}^n$, $B^{jk}\in\mathbb{R}^{n \times n}$,
with initial data $\tilde U(x,0)=\tilde U_0(x)$ and boundary
conditions as specified in (B) below.

An fundamental question is to establish {\it asymptotic stability}
of these solutions under perturbation of the initial or boundary
data. This question has been investigated in
\cite{GR,MeZ1,GMWZ5,GMWZ6,YZ,NZ1,NZ2} for arbitrary-amplitude
boundary-layers using Evans function techniques, with the result
that linearized and nonlinear stability reduce to a generalized
spectral stability, or Evans stability, condition. See also the
small-amplitude results of \cite{GG,R3,MN,KNZ,KaK} obtained by
energy methods.

In the current paper, as in \cite{N1} for the shock cases, we apply
the method of Kreiss' symmetrizers to provide a different proof of
estimates on low-frequency part of the solution operator, which
allows us to extend the existing stability result in \cite{NZ2} to a
larger class of symmetrizable systems including MHD equations,
yielding the result for certain MHD layers. We are also able to drop
a technical assumption (H4) that was required in previous analysis
of \cite{Z2,Z3,Z4,GMWZ1,NZ2}.

\subsection{Equations and assumptions}
We consider the general hyperbolic-parabolic system of conservation
laws \eqref{sys} in conserved variable $\tilde U$, with
$$\tilde U = \begin{pmatrix}\tilde u^I\\
\tilde u^{II}\end{pmatrix}, \quad B=\begin{pmatrix}0 & 0 \\
b^{jk}_1 & b^{jk}_2\end{pmatrix},$$ $\tilde u^I\in \RR^{n-r}$, $\tilde
u^{II}\in \RR^{r}$, and $$ \Re\sigma \sum_{jk} b_2^{jk}\xi_j\xi_k \ge
\theta |\xi|^2>0, \quad \forall \xi \in \RR^n\backslash \{0\}.$$

Following \cite{MaZ3,MaZ4,Z3,Z4}, we assume that equations
\eqref{sys} can be written,
alternatively, after a triangular change of coordinates
\begin{equation}\label{Wcoord}
\tilde W:=\tilde W(\tilde U) =\begin{pmatrix}\tilde w^I(\tilde u^I)\\
 \tilde w^{II}(\tilde u^I, \tilde u^{II})\end{pmatrix},
\end{equation}
in {\em the quasilinear, partially symmetric hyperbolic-parabolic
form}
\begin{equation}\label{symmetric}\tilde A^0 \tilde W_t + \sum_j\tilde A^j\tilde W_{x_j} = \sum_{jk}(\tilde
B^{jk} \tilde W_{x_k})_{x_j} + \tilde G,
\end{equation}
where, defining $\tilde W_+:=\tilde W(U_+)$,
\medskip

(A1) $\tilde A^j(\tilde W_+),\tilde A^0,\tilde A^1_{11}$ are
symmetric, $\tilde A^0$ block diagonal, $\tilde A^0\ge \theta_0>0$,
\medskip

(A2) for each $\xi \in \RR^d\setminus \{0\}$, no eigenvector of
$\sum_j\xi_j\tilde A^j(\tilde A^0)^{-1}(\tilde W_+)$ lies in the
kernel of $\sum_{jk}\xi_j\xi_k\tilde B^{jk}(\tilde A^0)^{-1}(\tilde
W_+)$,
\medskip

(A3) $\tilde B^{jk}=\begin{pmatrix}0 & 0 \\ 0 & \tilde
b^{jk}\end{pmatrix}$, $\sum\tilde b^{jk}\xi_j\xi_k\ge
\theta|\xi|^2$, and $\tilde G=\begin{pmatrix}0\\\tilde g
\end{pmatrix}$ with $\tilde g(\tilde W_x,\tilde W_x)=\cO(|\tilde
W_x|^2).$
\medskip

Along with the above structural assumptions, we make the following
technical hypotheses:
\medskip

(H0) $F^j, B^{jk}, \tilde A^0, \tilde A^j, \tilde B^{jk}, \tilde
W(\cdot), \tilde g(\cdot,\cdot) \in C^{s}$, with
$s\ge [(d-1)/2]+5$ in our analysis of linearized stability, and
$s\ge s(d):=[(d-1)/2]+7$ in
our analysis of nonlinear stability.
\medskip

(H1) $\tilde A_1^{11}$ is either strictly positive or strictly
negative, that is, either $\tilde A_1^{11}\ge \theta_1>0,$ or
$\tilde
A^{11}_1\le -\theta_1<0$. (We shall 
call these cases the {\em inflow case} or {\em outflow case},
correspondingly.)
\medskip

(H2) The eigenvalues of $dF^1(U_+)$ are
distinct and nonzero.
\medskip

(H3) The eigenvalues of $\sum_j \xi_jdF^j(U_+)$ have constant
multiplicity with respect to $\xi\in \RR^d$, $\xi\ne 0$.

\medbreak

\textbf{Alternative Hypothesis H3$'.$ } The constant multiplicity
condition in Hypothesis (H3) holds for the compressible Navier–
Stokes equations whenever  is hyperbolic. We are able to treat
symmetric dissipative systems like the equations of viscous MHD, for
which the constant multiplicity condition fails, under the following
relaxed hypothesis. \medbreak

(H3$'$) The eigenvalues of $\sum_j \xi_jdF^j(U_+)$ are either
semisimple and of constant multiplicity or totally nonglancing in
the sense of \cite{GMWZ6}, Definition 4.3. \medbreak

\textbf{Additional Hypothesis H4$'$ (in 3D).} In the treatment of
the three--dimensional case, the analysis turns out to be 
quite delicate and we are able to establish the stability under the
following additional (generic) hypothesis (see Remark \ref{rem-H4'}
and Appendix \ref{genH4} for discussions of this condition):

\medbreak

 (H4$'$) In the case the eigenvalue $\lambda_k(\xi)$ of
$\sum_j
\xi_jdF^j(U_+)$ is semisimple and of constant multiplicity, we assume further that $\nabla_{\txi}\lambda_k \not =0$ 
when $\partial_{\xi_1} \lambda_k=0$, $\xi \not=0$.


\begin{remark} \textup{Here we stress that we are able to
drop the following structural assumption, which is needed for the
earlier analyses of \cite{Z2,Z3,Z4,NZ2}.}\medbreak

\textup{(H4) The set of branch points of the eigenvalues of $(\tilde
A^1)^{-1}(i\tau \tilde A^0+ \sum_{j\ne 1} i\xi_j\tilde A^j)_+$,
$\tau \in \RR$, $\tilde \xi\in \RR^{d-1}$ is the (possibly
intersecting) union of finitely many smooth curves
$\tau=\eta_q^+(\tilde \xi)$, on which the branching eigenvalue has
constant multiplicity $s_q$ (by definition $\ge 2$).}
\end{remark}

We also assume:

(B) Dirichlet boundary conditions in $\tilde W$-coordinates:
\begin{equation}\label{inBC}
(\tilde w^I, \tilde w^{II})(0,\tx,t)=\tilde h(\tx,t):=(\tilde
h_1,\tilde h_2)(\tx,t)\end{equation} for the inflow case, and
\begin{equation}\label{outBC}
\tilde w^{II}(0,\tx,t)=\tilde h(\tx,t)\end{equation} for the outflow
case, with $x = (x_1,\tx)\in \RR^d$.
\\

This is sufficient for the main physical applications; the situation
of more general, Neumann and mixed-type boundary conditions on the
parabolic variable $\tilde w^{II}$ can be treated as discussed in
\cite{GMWZ5,GMWZ6}.

\subsection{The Evans condition and strong spectral stability}
A necessary condition for linearized stability is weak spectral
stability, defined as nonexistence of unstable spectra $\Re \lambda
>0$ of the linearized operator $L$ about the wave. As described in
\cite{Z2,Z3}, this is equivalent to nonvanishing for all $\tilde
\xi\in \RR^{d-1}$, $\Re \lambda>0$ of the {\it Evans function}
$$
D_L(\tilde \xi, \lambda),
$$a Wronskian associated with the family of eigenvalue ODE obtained by
Fourier transform in the directions $\tilde x:=(x_2,\dots, x_d)$.
See \cite{Z2,Z3,GMWZ5,GMWZ6,NZ2} for further discussion.

\begin{definition}\label{strongspectral}
\textup{ We define {\it strong spectral stability} as {\it uniform
Evans stability}:
\begin{equation}\tag{D}
|D_L(\tilde \xi, \lambda)|\ge \theta(C)>0
\end{equation}
for $(\tilde \xi, \lambda)$ on bounded subsets $C\subset \{\tilde
\xi\in \RR^{d-1}, \, \Re \lambda \ge 0\}\setminus\{0\}$. }
\end{definition}

For the class of equations we consider, this is equivalent to the
uniform Evans condition of \cite{GMWZ5,GMWZ6}, which includes an
additional high-frequency condition that for these equations is
always satisfied (see Proposition 3.8, \cite{GMWZ5}). A fundamental
result proved in \cite{GMWZ5} is that small-amplitude
noncharacteristic boundary-layers are always strongly spectrally
stable.

\begin{prop}[\cite{GMWZ5}]\label{specstab}
Assuming (A1)-(A3), (H0)-(H2), (H3$'$), (B) for some fixed endstate
(or compact set of endstates) $U_+$, boundary layers with amplitude
$$
\|\bar U-U_+\|_{L^\infty[0,+\infty]}
$$
sufficiently small satisfy the strong spectral stability condition
(D).
\end{prop}

As demonstrated in \cite{SZ,Z5}, stability of large-amplitude
boundary layers may fail for the class of equations considered here,
even in a single space dimension, so there is no such general
theorem in the large-amplitude case. Stability of large-amplitude
boundary-layers may be checked efficiently by numerical Evans
computations; see, e.g., \cite{HLZ,CHNZ,HLyZ1,HLyZ2}.

\subsection{Main results}
Our main results are as follows.

\begin{theorem}[Linearized stability]\label{theo-lin}
Assuming (A1)-(A3), (H0)-(H2), (H3$'$), (H4$'$), (B), and (D), we
obtain the asymptotic $L^1\cap H^{[(d-1)/2]+2} \rightarrow L^p$
stability in dimensions $d\ge 3$, and any $2\le p\le \infty$, with
rates of decay
\begin{equation}\begin{aligned} |U(t)|_{L^2}&\le C
(1+t)^{-\frac{d-2}{4}-\epsilon} |U_0|_{L^1\cap L^2},\\
|U(t)|_{L^p}&\le C (1+t)^{-\frac{d-1}{2}(1-1/p) +
\frac1{2p}-\epsilon} |U_0|_{L^1\cap H^{[(d-1)/2]+2}},
\end{aligned}\end{equation}
for some $\epsilon>0$, provided that the initial perturbations $U_0$
are in $L^1 \cap H^{[(d-1)/2]+2}$, and zero boundary perturbations.
\end{theorem}

\begin{theorem}[Nonlinear stability]\label{theo-nonlin}
Assuming (A1)-(A3), (H0)-(H2), (H3$'$), (H4$'$), (B), and (D), we
obtain the asymptotic $L^1\cap H^s \rightarrow L^p \cap H^s$
stability in dimensions $d\ge 3$, for $s\ge s(d)$ as defined in
(H0), and any $2\le p\le \infty$, with rates of decay
\begin{equation}\begin{aligned} &|\tilde U(t)-\bU|_{L^p}\le C
(1+t)^{-\frac{d-1}{2}(1-1/p) + \frac1{2p}-\epsilon} |U_0|_{L^1\cap
H^s}\\&|\tilde U(t)-\bU|_{H^s}\le C (1+t)^{-\frac{d-2}4-\epsilon}
|U_0|_{L^1\cap H^s},
\end{aligned}\end{equation}for some $\epsilon>0$,
provided that the initial perturbations $U_0:=\tilde U_0 - \bU$ are
sufficiently small in $L^1 \cap H^s$ and  zero boundary
perturbations.
\end{theorem}

\begin{remark} \textup{As will be seen in the proof, the assumption
(H4$'$) can be dropped in the case $d\ge 4$, though 
we then lose the factor $t^{-\epsilon}$ in the decay rate.}
\end{remark}

Our final main result gives the stability 
for the two--dimensional case that is not covered by the above
theorems. We remark here that as shown in \cite{Z2,Z3}, Hypothesis
(H4) is automatically satisfied in dimensions $d = 1, 2$ and in any
dimension for rotationally invariant problems. Thus, in treating the
two--dimensional case, we assume this hypothesis without making any
further restriction on structure of the systems. Also since the
proof does not depend on dimension $d$, we state the theorem in a
general form as follows.

\begin{theorem}[Two-dimensional case or cases with (H4)]\label{theo-stabH4}
Assume (A1)-(A3), (H0)-(H2), (H3$'$), (H4), (B), and (D). We obtain
asymptotic $L^1\cap H^s \rightarrow L^p \cap H^s$ stability of $\bar
U$ as a solution of \eqref{sys} in dimension $d\ge 2$, for $s\ge
s(d)$ as defined in (H0), and any $2\le p\le \infty$, with rates of
decay
\begin{equation}\begin{aligned} &|\tilde U(t)-\bU|_{L^p}\le C
(1+t)^{-\frac{d}{2}(1-1/p)+1/2p} |U_0|_{L^1\cap H^s}\\&|\tilde
U(t)-\bU|_{H^s}\le C (1+t)^{-\frac{d-1}4} |U_0|_{L^1\cap H^s},
\end{aligned}\end{equation}provided that the initial perturbations $U_0:=\tilde U_0 - \bU$ are
sufficiently small in $L^1 \cap H^s$ and zero boundary
perturbations. Similar statement holds for linearized stability.
\end{theorem}

%

\begin{remark}\textup{ The same results can be also obtained for nonzero
boundary perturbations as treated in \cite{NZ2}. In fact, in
\cite{NZ2}, though a bit of tricky, it has been already shown that
estimates on solution operator (see Proposition \ref{prop-estS}) for
homogenous boundary conditions are enough to treat nonzero boundary
perturbations. Thus for sake of simplicity, we only treat zero
boundary perturbations in the current paper.}
\end{remark}

Combining Theorems \ref{theo-lin}, \ref{theo-nonlin},
\ref{theo-stabH4} and Proposition \ref{specstab}, we obtain the
following small-amplitude stability result.

\begin{cor}\label{smallamp}
Assuming (A1)-(A3), (H0)-(H2), (H3$'$), (B) for some fixed endstate
(or compact set of endstates) $U_+$, boundary layers with amplitude
$$
\|\bar U-U_+\|_{L^\infty[0,+\infty]}
$$
sufficiently small are linearly and nonlinearly stable in the sense
of Theorems \ref{theo-lin}, \ref{theo-nonlin}, and
\ref{theo-stabH4}.
\end{cor}


\section{Nonlinear stability}
The linearized equations of \eqref{sys} about the profile $\bar U$ are
\begin{equation}\label{lind}
 U_t = LU :=
\sum_{j,k}( B^{jk} U_{x_k})_{x_j}
- \sum_{j}( A^{j} U)_{x_j}
\end{equation}
with initial data $U(0)=U_0$. Then, we obtain the following
proposition, extending Proposition 3.5 of \cite{NZ2} under our
weaker assumptions.
\begin{proposition}\label{prop-estS} Under the hypotheses of Theorem \ref{theo-nonlin},
the solution operator $\cS(t):=e^{Lt}$ of the linearized equations
may be decomposed into low frequency and high frequency parts (see
below) as $\cS(t) = \cS_1(t) + \cS_2(t)$ satisfying
\begin{equation}\label{boundcS1}
\begin{aligned} |\cS_1(t) \partial_{x_1}^{\beta_1}\partial_{\tx}^{\tilde \beta} f|_{L^2_x} \le& C
(1+t)^{-(d-2)/4-\epsilon/2- |\beta|/2}|f|_{L^1_x}+C
(1+t)^{-(d-2)/4-\epsilon/2}|f|_{L^{1,\infty}_{\tx,x_1}}\\
|\cS_1(t) \partial_{x_1}^{\beta_1}\partial_{\tx}^{\tilde \beta}
f|_{L^{{2,\infty}}_{\tx,x_1}} \le& C (1+t)^{-(d-1)/4 -\epsilon/2-
|\beta|/2}|f|_{L^1_x}+C(1+t)^{-(d-1)/4-\epsilon/2}|f|_{L^{1,\infty}_{\tx,x_1}}\\
|\cS_1(t) \partial_{x_1}^{\beta_1}\partial_{\tx}^{\tilde \beta}
f|_{L^{\infty}_{x}} \le& C (1+t)^{-(d-1)/2-\epsilon/2- |\beta|/2}|f|_{L^1_x}+C
(1+t)^{-(d-1)/2-\epsilon/2}|f|_{L^{1,\infty}_{\tx,x_1}}
\end{aligned}\end{equation}
for some $\epsilon>0$ and $\beta = (\beta_1,\tilde\beta)$ with $\beta_1=0,1$, and
\begin{equation}\label{boundcS2}
\begin{aligned}|\partial_{x_1}^{\gamma_1}\partial_{\tx}^{\tilde\gamma}\cS_2(t)f|_{L^2} &\le C
e^{-\theta_1t}|f|_{H^{|\gamma_1|+|\tilde \gamma|+3}},\end{aligned}\end{equation} for $\gamma = (\gamma_1,\tilde\gamma)$ with $\gamma_1=0,1$.
\end{proposition}

We shall give a proof of Proposition \ref{prop-estS} in Section
\ref{sec-estS}. For the rest of this section, we give a rather
straightforward proof of the first two main theorems using estimates
of the solution operator stated in Proposition \ref{prop-estS},
following nonlinear arguments of \cite{Z3,NZ2}.

\subsection{Proof of linearized stability} Applying estimates on low- and high-frequency operators $\cS_1(t)$
and $\cS_2(t)$ obtained in Proposition \ref{prop-estS}, we obtain
\begin{equation} \begin{aligned} |U(t)|_{L^2} &\le
|\cS_1(t)U_0|_{L^2} + |\cS_2(t)U_0|_{L^2}\\&\le
C(1+t)^{-\frac{d-2}{4}-\frac
\epsilon2}[|U_0|_{L^1}+|U_0|_{L^{1,\infty}_{\tx,x_1}}] + Ce^{-\eta
t}|U_0|_{H^3}\\&\le C(1+t)^{-\frac{d-2}{4}-\frac
\epsilon2}|U_0|_{L^1\cap H^3}
\end{aligned}\end{equation}
and (together with Sobolev embedding) \begin{equation}
\begin{aligned} |U(t)|_{L^\infty} &\le |\cS_1(t)U_0|_{L^\infty} +
|\cS_2(t)U_0|_{L^\infty}\\&\le C(1+t)^{-\frac{d-1}{2}-\frac
\epsilon2}[|U_0|_{L^1} +|U_0|_{L^{1,\infty}_{\tx,x_1}}]+
C|\cS_2(t)U_0|_{H^{[(d-1)/2]+2}}\\&\le C(1+t)^{-\frac{d-1}{2}-\frac
\epsilon2}[|U_0|_{L^1} +|U_0|_{L^{1,\infty}_{\tx,x_1}}]+ Ce^{-\eta
t}|U_0|_{H^{[(d-1)/2]+2}}\\&\le C(1+t)^{-\frac{d-1}{2}-\frac
\epsilon2}|U_0|_{L^1\cap H^{[(d-1)/2]+2}}.
\end{aligned}\end{equation}
These prove the bounds as stated in the theorem for $p=2$ and
$p=\infty$. For $2<p<\infty$, we use the interpolation inequality
between $L^2$ and $L^\infty$.

\subsection{Proof of nonlinear stability} Defining the perturbation variable $U:= \tilde U - \bU$, we obtain
the nonlinear perturbation equations
\begin{equation}\label{per-eqs} U_t - LU = \sum_j
Q^j(U,U_x)_{x_j},\end{equation} where
\begin{equation}\label{newqbounds}
\begin{aligned}
Q^j(U,U_x)&=\cO(|U||U_x|+|U|^2)\\
Q^j(U,U_x)_{x_j}&= \cO(|U||U_{x}|+|U||U_{xx}|+|U_x|^2)\\
Q^j(U,U_x)_{x_jx_k}&=
\cO(|U||U_{xx}|+|U_x||U_{xx}|+|U_x|^2+|U||U_{xxx}|)
\end{aligned}
\end{equation}
so long as $|U|$ remains bounded.

Applying the Duhamel principle to \eqref{per-eqs}, we obtain
\begin{equation}\label{Duhamel}
\begin{aligned}
  U(x,t)=& \cS(t) U_0 + \int_0^t \cS(t-s)\sum_j
\partial_{x_j}Q^j(U,U_x)ds
\end{aligned}
\end{equation} where $U(x,0) = U_0(x)$.

\begin{proof}[Proof of Theorem \ref{theo-nonlin}]
Define \begin{equation}\label{zeta} \begin{aligned}\zeta(t):=\sup_s
&\Big(|U(s)|_{L^2_x}(1+s)^{\frac{d-2}4+\epsilon}+|U(s)|_{L^\infty_x}(1+s)^{\frac
{d-1}2+\epsilon}
\\&+(|U(s)|+|U_x(s)|)_{L^{2,\infty}_{\tx,x_1}}(1+s)^{\frac{d-1}4+\epsilon}
\Big).\end{aligned}
\end{equation}

 We shall prove here that for all $t\ge
0$ for which a solution exists with $\zeta(t)$ uniformly bounded by
some fixed, sufficiently small constant, there holds
\begin{equation}\label{zeta-est}
\zeta(t) \le C(|U_0|_{L^1\cap H^s}+\zeta(t)^2) .\end{equation}

This bound together with continuity of $\zeta(t)$ implies that
\begin{equation}\label{zeta-est1} \zeta(t) \le 2C|U_0|_{L^1\cap H^s}\end{equation}
for $t\ge0$, provided that $|U_0|_{L^1\cap H^s}< 1/4C^2$. This would
complete the proof of the bounds as claimed in the theorem, and thus
give the main theorem.

By standard short-time theory/local well-posedness in $H^s$, and the
standard principle of continuation, there exists a solution $U\in
H^s$ on the open time-interval for which $|U|_{H^s}$ remains
bounded, and on this interval $\zeta(t)$ is well-defined and
continuous. Now, let $[0,T)$ be the maximal interval on which
$|U|_{H^s}$ remains strictly bounded by some fixed, sufficiently
small constant $\delta>0$. Recalling the following energy estimate
(see Proposition 4.1 of \cite{NZ2}) and the Sobolev embeding
inequality $|U|_{W^{2,\infty}}\le C|U|_{H^s}$, we have
\begin{equation}\label{Hs}\begin{aligned}|U(t)|_{H^s}^2 &\le Ce^{-\theta t}|U_0|_{H^s}^2
+ C \int_0^t e^{-\theta(t-\tau)}|U(\tau)|_{L^2}^2d\tau\\&\le
C(|U_0|_{H^s}^2 +\zeta(t)^2)(1+t)^{-(d-2)/2-2\epsilon}.
\end{aligned}\end{equation}
and so the solution continues so long as $\zeta$ remains small, with
bound \eqref{zeta-est1}, yielding existence and the claimed bounds.

Thus, it remains to prove the claim \eqref{zeta-est}. First by
\eqref{Duhamel}, we obtain
\begin{equation}\begin{aligned} |U(t)|_{L^2}\le& |\cS(t)U_0|_{L^2} +
\int_0^t|\cS_1(t-s)\partial_{x_j}Q^j(s)|_{L^2}ds \\&+ \int_0^t
|\cS_2(t-s)\partial_{x_j}Q^j(s)|_{L^2}ds
\end{aligned}\end{equation}
where $|\cS(t) U_0|_{L^2}\le C
(1+t)^{-\frac{d-1}{4}-\epsilon}|U_0|_{L^1\cap H^3}$ and
$$\begin{aligned}
\int_0^t|&\cS_1(t-s)\partial_{x_j}Q^j(s)|_{L^2}ds
\\&\le C\int_0^t (1+t-s)^{-\frac{d-2}{4}-\frac12-\epsilon}|Q^j(s)|_{L^1} + (1+s)^{-\frac{d-2}4-\epsilon}|Q^j(s)|_{L^{1,\infty}_{\tx,x_1}} ds
\\&\le C\int_0^t (1+t-s)^{-\frac{d-2}{4}-\frac12-\epsilon}|U|_{H^1}^2
+
(1+t-s)^{-\frac{d-2}4-\epsilon}\Big(|U|^2_{L^{2,\infty}_{\tx,x_1}}+|U_x|^2_{L^{2,\infty}_{\tx,x_1}}\Big)ds\\&\le
C(|U_0|_{H^s}^2+\zeta(t)^2)\int_0^t \Big[
(1+t-s)^{-\frac{d-2}{4}-\frac12-\epsilon}(1+s)^{-\frac{d-2}{2}-2\epsilon}
\\&\quad\qquad+ (1+t-s)^{-\frac{d-2}4-\epsilon}(1+s)^{-\frac{d-1}2-2\epsilon}\Big]ds\\&\le
C(1+t)^{-\frac{d-2}{4}-\epsilon}(|U_0|_{H^s}^2+\zeta(t)^2)
\end{aligned}$$
and
$$\begin{aligned}\int_0^t
|&\cS_2(t-s)\partial_{x_j}Q^j(s)|_{L^2}ds\\&\le \int_0^t
e^{-\theta(t-s)}|\partial_{x_j}Q^j(s)|_{H^3}ds
\\&\le C\int_0^t
e^{-\theta(t-s)}|U|_{H^s}^2ds
\\&\le C(|U_0|_{H^s}^2+\zeta(t)^2) \int_0^t
e^{-\theta(t-s)}(1+s)^{-\frac{d-2}{2}-2\epsilon}ds\\&\le
C(1+t)^{-\frac{d-2}{2}-2\epsilon}(|U_0|_{H^s}^2+\zeta(t)^2).
\end{aligned}$$

Therefore, combining these above estimates yields \begin{equation}
|U(t)|_{L^2}(1+t)^{\frac{d-2}{4}+\epsilon} \le C(|U_0|_{L^1\cap
H^s}+\zeta(t)^2) .\end{equation}

Similarly, we can obtain estimates for other norms of $U$ in
definition of $\zeta$, and finish the proof of claim
\eqref{zeta-est} and thus the main theorem.\end{proof}

\begin{remark}\label{rem-est-refined} \textup{The decaying factor $t^{-\epsilon}$ is crucial in above analysis when
$d=3$. In fact, the main difficulty here comparing with the shock
cases in \cite{N1} is to obtain a refined bound of solutions in
$L^\infty$. See further discussion in Section \ref{sec-estS} below.}
\end{remark}

\section{Linearized estimates}\label{sec-estS}
In this section, we shall give a proof of Proposition
\ref{prop-estS} or bounds on $\cS_1(t)$ and $\cS_2(t)$, where we use
the same decomposition of solution operator $\cS(t) =
\cS_1(t)+\cS_2(t)$ as in \cite{Z2,Z3}.

\subsection{High--frequency estimate} We first observe that our relaxed
Hypothesis (H3$'$) and the dropped Hypothesis (H4) only play a role
in low--frequency regimes. Thus, in course of obtaining the
high--frequency estimate \eqref{boundcS2}, we make here the same
assumptions as were made in \cite{NZ2}, and therefore the same
estimate remains valid as claimed in \eqref{boundcS2} under our
current assumptions. We omit to repeat its proof here, and refer the
reader to the paper \cite{NZ2}, Proposition 3.6.

In the remaining of this section, we shall focus on proving the
bounds on low-frequency part $\cS_1(t)$ of linearized solution
operator.

Taking the Fourier transform in $\tilde x:=(x_2,\dots,x_d)$ of
linearized equation \eqref{lind}, we obtain a family of eigenvalue
ODE
\begin{equation}\label{eigensys}\begin{aligned}
\lambda U=L_{\tilde \xi }U:=\overbrace{(B_{11}U')'-(A_1U)'}^{L_0U} -
&i\sum_{j\not=1}A_j\xi_jU + i\sum_{j\not=1}B_{j1}\xi_jU'
\\&+i\sum_{k\not=1}(B_{1k}\xi_kU)' -
\sum_{j,k\not=1}B_{jk}\xi_j\xi_k U.
\end{aligned}\end{equation}

\subsection{The GMWZ's $L^2$ stability estimate}\label{GMWZresults} Let $U=(u^I,u^{II})^T$ be a solution of
resolvent equation $(L_\txi-\lambda)U = f$. Following
\cite{Z3,GMWZ6}, consider the variable $W$ as
usual$$W:=\begin{pmatrix}w^{I}\\w^{II}\\w^{II}_{x_1}\end{pmatrix}$$
with $w^I:=A_*u^I, w^{II}:=b_1^{11}u^I + b^{11}_2u^{II}$,
$A_*:=A^1_{11} - A^1_{12}(b_2^{11})^{-1}b_1^{11}$. Then we can write
equations of $W$ as a first order system
\begin{equation}\label{1steqsW} \begin{aligned}\partial_{x_1} W &= \cG(x_1,\lambda,\txi)W+F
\\\Gamma W &=0 \mbox{  on  }x_1=0.\end{aligned}\end{equation}

For small or bounded frequencies $(\lambda,\txi)$, we use the MZ
conjugation lemma (see \cite{MeZ1,MeZ3}). That is, given any
$(\underline{\lambda},\underline{\txi})\in \RR^{d+1}$, there is a
smooth invertible matrix $\Phi(x_1,\lambda,\txi)$ for $x_1\ge 0$ and
$(\lambda,\txi)$ in a small neighborhood of
$(\underline{\lambda},\underline{\txi})$, such that \eqref{1steqsW}
is equivalent to
\begin{equation}\label{1steqsY} \partial_{x_1} Y = \cG_+(\lambda,\txi)Y+\tilde F, \quad \tilde \Gamma(\lambda,\txi)Y= 0\end{equation}
where $\cG_+(\lambda,\txi) := \tilde \cG(+\infty,\lambda,\txi), W =
\Phi Y, \tilde F =\Phi^{-1}F $ and $\tilde \Gamma Y:= \Gamma\Phi Y$.

Next, there are smooth matrices $V(\lambda,\txi)$ such that
\begin{equation} V^{-1}\cG_+ V =
\begin{pmatrix}H&0\\0&P\end{pmatrix}\end{equation} with blocks
$H(\lambda,\txi)$ and $P(\lambda,\txi)$ satisfying the eigenvalues
$\mu$ of $P$ in $\{|\R\mu|\ge c>0\}$ and
$$\begin{aligned}H(\lambda,\txi) &= H_0(\lambda,\txi) +
\cO(\rho^2)\\H_0(\lambda,\txi):&=-(A_{+}^1)^{-1}\Big((i\tau +
\gamma)A^0_+ + \sum_{j=2}^di\xi_jA^j_+\Big),\end{aligned}$$ with
$\lambda = \gamma + i\tau.$

Define variables $Z=(u_H,u_P)^T$ as $ W = \Phi Y = \Phi V Z$, $\bar
\Gamma Z:= \Gamma \Phi V Z,$ and $ (f_H,f_P)^T = V^{-1}\tilde F$. We
have
\begin{equation}\label{1steqsZ}\partial_{x_1}\begin{pmatrix}u_H\\u_P\end{pmatrix}
= \begin{pmatrix}H&0\\0&P\end{pmatrix}
\begin{pmatrix}u_H\\u_P\end{pmatrix} +
\begin{pmatrix}f_H\\f_P\end{pmatrix}, \quad \bar \Gamma Z =
0.\end{equation}

Then the maximal stability estimate for the low frequency regimes in
\cite{GMWZ6} states that
\begin{equation}\label{max-est} (\gamma+\rho^2)|u_H|^2_{L^2}+|u_P|^2_{L^2} + |u_H(0)|^2+|u_P(0)|^2
\lesssim \wprod {|f_H|,|u_H|}+\wprod{|f_P|,|u_P|}.\end{equation}

We note that in the final step there in \cite{GMWZ1}, the standard
Young's inequality has been used to absorb all terms of $(u_H,u_P)$
into the left-hand side, leaving the $L^2$ norm of $F$ alone in the
right hand side. For our purpose, we shall keep it as stated in
\eqref{max-est}. Here, by $f \lesssim g$, we mean $f \le C g$, for
some $C$ independent of parameter $\rho$.

We remark also that as shown in \cite{GMWZ1}, all of coordinate
transformation matrices are uniformly bounded. Thus a bound on $Z =
(u_H,u_P)^T$ would yield a corresponding bound on the solution $U$.

\subsection{$L^2$ and $L^\infty$ resolvent bounds}
Changing variables as above and taking the inner product of each
equation in \eqref{1steqsZ} against $u_H$ and $u_P$, respectively,
and integrating the results over $[0,x_1]$, for $x_1>0$, we obtain
\begin{equation}\begin{aligned} \frac 12 |u_H(x_1)|^2 &= \frac 12|u_H(0)|^2 + \R\int_0^{x_1}(H(\lambda,\txi)u_H\cdot u_H  + f_H\cdot u_H)dz,\\\frac 12 |u_P(x_1)|^2 &= \frac 12|u_P(0)|^2 + \R\int_0^{x_1}(P(\lambda,\txi)u_P\cdot u_P  + f_P\cdot u_P)dz.\end{aligned}\end{equation}

This together with the facts that $|H|\le C\rho$ and $|P|\le C$
yields
\begin{equation}\begin{aligned} |u_H|_{L^\infty(x_1)}^2 &\lesssim |u_H(0)|^2 + \rho|u_H|_{L^2}^2 + \wprod {|f_H|,|u_H|},
\\|u_P|_{L^\infty(x_1)}^2 &\lesssim |u_P(0)|^2 + |u_P|_{L^2}^2+\wprod{|f_P|,|u_P|},\end{aligned}\end{equation}
and thus in view of \eqref{max-est} gives
\begin{equation} \label{Linftyest}(\gamma+\rho^2)|u_H|^2_{L^2}+|u_P|^2_{L^2} + \rho|u_H|_{L^\infty}^2+|u_P|_{L^\infty}^2\lesssim \wprod {|f_H|,|u_H|}+\wprod{|f_P|,|u_P|}.\end{equation}

Now applying the Young's inequality, we get $$\wprod
{|f_H|,|u_H|}+\wprod{|f_P|,|u_P|} \le (\epsilon |u_P|_{L^\infty}^2 +
C_\epsilon|f_P|_{L^1}^2) +\Big(\epsilon \rho|u_H|_{L^\infty}^2 +
\frac{C_\epsilon}{\rho}|f_H|_{L^1}^2\Big) $$and thus for $\epsilon$
sufficiently small, together with \eqref{Linftyest},
\begin{equation}(\gamma+\rho^2)|u_H|^2_{L^2}+|u_P|^2_{L^2} + \rho|u_H|_{L^\infty}^2+|u_P|_{L^\infty}^2\lesssim \frac{1}{\rho}|f_H|_{L^1}^2 + |f_P|_{L^1}^2.\end{equation}

Therefore in term of $Z = (u_H,u_P)^t$,
\begin{equation}\label{estZ}\begin{aligned} |Z|_{L^\infty(x_1)}\le
C\rho^{-1}|f|_{L^1} \quad\mbox{and}\quad |Z|_{L^2(x_1)} \le
C\rho^{-3/2}|f|_{L^1}.\end{aligned}\end{equation}

Unfortunately, unlike the shock cases (see \cite{N1}), bounds
\eqref{estZ} are not enough for our need to close the analysis in
dimension $d=3$. See Remark \ref{rem-est-refined}. In the following
subsection, we shall derive better bounds for $Z$ in both $L^\infty$
and $L^2$ norms.

\subsection{Refined $L^2$ and $L^\infty$ resolvent bounds}\label{sec-refinedbd} With the same
notations as above, we prove in this subsection that there hold refined resolvent bounds:
\begin{equation}\label{estZ-refined}\begin{aligned} |Z|_{L^\infty(x_1)}\lesssim
\rho^{-1+\epsilon}(|f|_{L^1}+|f|_{L^\infty})\qquad \mbox{and}\qquad
|Z|_{L^2(x_1)}\lesssim
\rho^{-3/2+\epsilon}(|f|_{L^1}+|f|_{L^\infty})\end{aligned}\end{equation}
for some small $\epsilon>0$. We stress here that a refined factor
$\rho^\epsilon$ in $L^\infty$ 
is
crucial in our analysis for three-dimensional case. See Remark
\ref{rem-est-refined}.

Assumption (H3$'$) implies the following block structure (see
\cite{MeZ3,GMWZ6}). Here, we use the polar coordinate notation
$\zeta = (\tau,\gamma,\txi), \zeta = \rho \hat \zeta$, where $\hat
\zeta = (\hat \tau, \hat \gamma, \hat \txi)$ and $\hat \zeta \in
S^d$.

\begin{proposition}[Block structure; \cite{GMWZ6}]\label{prop-blockst} For all $\underline{\hat\zeta}$ with $\underline{\hat
\gamma} \ge 0$ there is a neighborhood $\omega$ of
$(\underline{\hat\zeta},0)$ in $S^d \times \overline{\RR}_+$ and
there are $C^\infty$ matrices $T(\hat \zeta, \rho)$ on $\omega$ such
that $T^{-1}H_0 T$ has the block diagonal structure
\begin{equation}
T^{-1}H_0 T = H_B(\hat \zeta, \rho) = \rho \hat H_B(\hat \zeta,
\rho)\end{equation} with
\begin{equation}
\label{block2b} \hat H_B(\hat \zeta, \rho)=
\left[\begin{array}{cccc} Q_1 & 0 & &
\\
  0    &      \ddots &0\\
 &     0 & Q_{p}
\end{array}\right]\, (\hat \zeta, \rho)
\end{equation}
 with  diagonal blocks $ Q_k $
of size $\nu_k \times \nu_k$ such that:

(i) (Elliptic modes) $\Re Q_k $ is either positive definite or
negative definite.

(ii) (Hyperbolic modes) $\nu _k =1$, $\Re Q_k =0$ when $\hat \gamma
= \rho =0$, and $\partial_{\cg}(\Re Q_k) \partial_{\rho}(\Re Q_k)
>0$.

(iii) (Glancing modes) $\nu _k >1$, $Q_k$ has the following form:
 \begin{equation} \label{Dbl1}
 Q_k ({\hat \zeta}, \rho)    =
i(\underline {\mu}_k \Id + J) + i\sigma Q'_k({\hat \txi}) + \cO(\hat
\gamma + \rho),
\end{equation}
where $\sigma:=|{\hat \txi}-\underline{\hat \txi}|,$
 \begin{equation}
 J:=\left[\begin{array}{cccc} 0  & 1 & 0&
\\
0  &0  & \ddots  &  0
\\
  & \ddots &      \ddots & 1 \\
 &  &  0 & 0
\end{array}\right], \qquad Q'_k({\hat \txi}):=\left[\begin{array}{cccc} q_1  & 0 & \cdots
& 0
\\
q_2  & 0 & \cdots & 0
\\
   &  & \cdots
& \\
 q_{\nu_k} & 0 & \cdots
& 0
\end{array}\right]
\end{equation}
$q_{\nu_k}\not=0$, and the lower left hand corner $a$ of $Q_k$
satisfies $\partial_{\cg}(\Re a)
\partial_{\rho}(\Re a)
>0$.

(iv) (Totally nonglancing modes) $\nu _k >1$, eigenvalue of $Q_k$,
when $\hat \gamma = \rho =0$, is totally nonglancing, see Definition
4.3, \cite{GMWZ6}.

\end{proposition}

\begin{proof} For a proof, see for example \cite{Met}, Theorem
8.3.1. It is also straightforward to see that for the case (iii),
$$q_{\nu_k}(\underline{\hat \txi}) = |\nabla_{\txi}D_k(\underline{\zeta},\underline{\xi}_1) | = c|\nabla_\txi \lambda_k(\xi)|,$$
where $c$ is a nonzero constant, $D_k(\zeta,\xi_1)$ is defined as $
\det (iQ_k(\zeta) + \xi_1 Id)$, and $\lambda_k(\xi)$ is the zero of
$D_k(\zeta,\xi_1)$ (recalling $\zeta = (\lambda,\txi)$) satisfying
$$
\partial_{\xi_1}\lambda_k =...=\partial_{\xi_1}^{\nu_k-1}\lambda_k
=0, \quad
\partial_{\xi_1}^{\nu_k} \lambda_k \not = 0\qquad \mbox{at  }(\underline{\txi},\underline{\xi}_1).$$
Thus, assumption (H4$'$) guarantees the nonvanishing of $q_{\nu_k}$.
We skip the proof of other facts.
\end{proof}

We shall treat each mode in turn. The following simple lemma may be
found useful.

\begin{lemma}\label{lem-bounds} Let $U$ be a solution of $\partial_z U = QU + F$ with $U(+\infty)=0$. Assume that there is a positive
[resp., negative] symmetric matrix $S$ such that
\begin{equation}\label{Sym-ineq}\Re SQ := \frac12 (SQ + Q^* S^*)\ge \theta
Id\end{equation} for some $\theta>0$, and $S\ge Id$ [resp., $-S\ge
Id$]. Then there holds
\begin{equation}\label{est-posneg}\begin{aligned} |U|^2_{L^\infty}+
\theta|U|_{L^2}^2 &\lesssim
|F|_{L^1}^2\\
\mbox{[resp.,   } |U|^2_{L^\infty} + \theta|U|_{L^2}^2&\lesssim
|U(0)|^2 + |F|_{L^1}^2\mbox{ ]}.\end{aligned}\end{equation}
\end{lemma}
\begin{proof} Taking the inner product of the
equation of $U$ against $SU$ and integrating the result over
$[x_1,\infty]$ for the first case [resp., $[0,x_1]$ for the second
case], we easily obtain the lemma.\end{proof}

Thanks to Proposition \ref{prop-blockst}, we can decompose $U$ as
follows
\begin{equation}U = u_{P} + u_{H_e} + u_{H_{h}} + u_{H_{g}}
+ u_{H_{t}},\end{equation}corresponding to parabolic, elliptic,
hyperbolic, glancing, or totally nonglancing modes.

\subsubsection{Parabolic modes} Since spectrum of $P$ is away from
the imaginary axis, we can assume that
$$P(\lambda,\txi)=\begin{pmatrix}P_+&0\\0&P_-\end{pmatrix}$$
with $\pm \Re P_\pm \ge c>0$. Therefore applying Lemma
\ref{lem-bounds} with $S = Id$ or $-Id$ yields
\begin{equation}\label{para-est}\begin{aligned} |u_{P_+}|^2_{L^\infty}+ |u_{P_+}|_{L^2}^2 &\lesssim |F_{P_+}|_{L^1}^2,
\\|u_{P_-}|^2_{L^\infty} + |u_{P_-}|_{L^2}^2&\lesssim |u_{P_-}(0)|^2
+ |F_{P_-}|_{L^1}^2.\end{aligned}\end{equation}

\subsubsection{Elliptic modes} This is case (i) in
Proposition \ref{prop-blockst} when the spectrum of $Q_k$ lies in
$$\{\R \mu > \delta \} \qquad [\mbox{resp.,  }\{\R \mu
<-\delta\}].$$

In this case, there are positive symmetric matrices $S^k(\hat
\zeta,\rho)$, $C^\infty$ on a neighborhood $\omega$ of
$(\underline{\hat \zeta},0)$ and such that
$$ \Re S^k Q^k \ge c Id \qquad [\mbox{resp.,  }-\Re S^k Q^k \ge c
Id]$$ for $c>0$. Thus, Lemma \ref{lem-bounds} again yields
\begin{equation}\label{elli-est}\begin{aligned} |u_{H_{e+}}|^2_{L^\infty}+ \rho|u_{H_{e+}}|_{L^2}^2 &\lesssim |F_{H_{e+}}|_{L^1}^2,
\\|u_{H_{e-}}|^2_{L^\infty} + \rho|u_{H_{e-}}|_{L^2}^2&\lesssim |u_{H_{e-}}(0)|^2
+ |F_{H_{e-}}|_{L^1}^2.\end{aligned}\end{equation}

\subsubsection{Hyperbolic modes} This is case (ii) in Proposition
\ref{prop-blockst}. In this case, as shown in \cite{Met} we can
write
\begin{equation} Q^k(\hat\zeta,\rho) = q^k(\hat \zeta) Id + \rho
\cR^k(\hat \zeta,\rho)\end{equation} where $q^k$ is purely imaginary
when $\hat \gamma =0$, $\dot{q}^k:=\partial_{\hat \gamma}\R q^k(\hat
\zeta)$ does not vanish, and the spectrum of $\dot{q}^k\cR^k(\hat
\zeta,0)$ is contained in the half space $\{\R\mu
>0\}$. Therefore, when $\dot q^k >0 $ [resp., $\dot q^k <0 $] and thus for
$(\zeta,\hat\gamma)$ sufficiently close to $(\hat \zeta,0)$
$$ \R q^k \ge c\hat \gamma, \qquad [\mbox{resp.,  }\R q^k \le -c \hat
\gamma],$$ we have positive symmetric matrices $S^k(\hat
\zeta,\rho)$ satisfying
$$ \Re S^k Q^k \ge c (\hat \gamma + \rho) Id \qquad [\mbox{resp.,  }-\Re S^k Q^k \ge c
(\hat \gamma + \rho)Id]$$ for $c>0$. Thus, again by Lemma
\ref{lem-bounds}, we obtain
\begin{equation}\label{hyper-est}\begin{aligned} |u_{H_{h+}}|^2_{L^\infty}+ (\gamma+\rho^2)|u_{H_{h+}}|_{L^2}^2 &\lesssim |F_{H_{h+}}|_{L^1}^2,
\\|u_{H_{h-}}|^2_{L^\infty} + (\gamma+\rho^2)|u_{H_{h-}}|_{L^2}^2&\lesssim |u_{H_{h-}}(0)|^2
+ |F_{H_{h-}}|_{L^1}^2.\end{aligned}\end{equation}

\subsubsection{Totally nonglancing modes}\label{subsec-totalnongl} This is case (iv) in
Proposition \ref{prop-blockst}. As constructed in \cite{GMWZ6},
there exist symmetrizers $S^k$ that are positive [resp. negative]
definite when the mode is totally incoming [resp. outgoing]. Denote
$u_{H_{t+}} $ [resp., $u_{H_{t-}}$] associated with totally incoming
[resp. outgoing] modes. Then similarly as in above, we also have
\begin{equation}\label{totalnongl-est}\begin{aligned} |u_{H_{t+}}|^2_{L^\infty}+ (\gamma+\rho^2)|u_{H_{t+}}|_{L^2}^2 &\lesssim |F_{H_{t+}}|_{L^1}^2,
\\|u_{H_{t-}}|^2_{L^\infty} + (\gamma+\rho^2)|u_{H_{t-}}|_{L^2}^2&\lesssim |u_{H_{t-}}(0)|^2
+ |F_{H_{t-}}|_{L^1}^2.\end{aligned}\end{equation}

Thus, putting these estimates together with noting that the
stability estimate \eqref{max-est} already gives a bound on
$|u(0)|$, we easily obtain sharp bounds on $u$ in $L^\infty$ and
$L^2$ for all above cases:
\begin{equation}\label{goodbound1} |u_k|_{L^\infty}^2 +\rho^2|u_k|_{L^2}^2 \lesssim
|f|_{L^1}^2 + |u_{H_g}|_{L^\infty}|f|_{L^1},\end{equation} for all
$k = P,H_e,H_h,H_t$.

\subsubsection{Glancing modes}  Hence, we remain to consider the final
case: case (iii) in Proposition \ref{prop-blockst}. Recall
\eqref{Dbl1}
 \begin{equation}\label{gl-block}
 Q_k ({\hat \zeta}, \rho)    =
i(\underline {\mu}_k \Id + J) + i\sigma Q'_k({\hat \txi}) + \cO(\hat
\gamma + \rho)
\end{equation} on a
neighborhood of $(\underline{\hat \zeta},0)$, where $\sigma=|{\hat
\txi}-\underline{\hat \txi}|$. We consider two cases.\\\\
{\bf Case a.} $\sigma \lesssim \rho^\epsilon$ for some small
$\epsilon
>0$. Recall that we consider the reduced system:
\begin{equation}\label{sys-uk} \partial_{x_1}u_k =\rho Q_k(\hat
\zeta,\rho) u_k + f_k\end{equation} with $Q_k(\hat \zeta,\rho)$
having a form as in \eqref{gl-block}. It is clear that the $L^p$
norm of $u_k$ remains unchanged under the transformation $u_k$ to $
u_k e^{-i\underline{\mu}_kx_1}$. Thus, we can assume that
$\underline{\mu}_k =0$. Note that we have the following bounds by
\eqref{estZ} \begin{equation}\label{est-uk}\begin{aligned}
|u_k|_{L^\infty(x_1)}\lesssim \rho^{-1}|f|_{L^1}
\quad\mbox{and}\quad |u_k|_{L^2(x_1)} \lesssim
\rho^{-3/2}|f|_{L^1}.\end{aligned}\end{equation}

To prove the refined bounds \eqref{estZ-refined}, we first observe
that
$$|\partial_{x_1}u_{k}|_{L^\infty} \lesssim \rho |u_k|_{L^\infty} + |f_k|_{L^\infty}
\lesssim |f|_{L^1} + |f|_{L^\infty},$$ where the last inequality is
due to \eqref{est-uk}. Now, write $u_k =
(u_{k,1},\cdots,u_{k,\nu_k})$. Thanks to the special form of $Q_k$
in \eqref{gl-block}, we have
\begin{equation}\label{sys-uks}\partial_{x_1}u_{k,\nu_k} =i\rho \sigma Q'_k(\hat\txi)u_k +\cO(\gamma + \rho^2)u_k+
f_k.
\end{equation}
Taking inner product of the equation \eqref{sys-uks} against
$\partial_{x_1}u_{k,\nu_k}$, we easily obtain by applying the
standard Young's inequality:
\begin{equation}\label{est-der-uks}|\partial_{x_1}u_{k,\nu_k}|_{L^2}^2 \lesssim \rho^{2+2\epsilon}|u_k|^{2}_{L^2}
+|f_k|_{L^1} |\partial_{x_1}u_{k,\nu_k}|_{L^\infty}\lesssim \rho^{-1+2\epsilon}|f|_{L^1}^2+|f|_{L^\infty}^2.
\end{equation}

Similarly, for $u_{k,\nu_k-1}$ satisfying
$$\partial_{x_1}u_{k,\nu_k-1} =i \rho \sigma Q'_k(\hat\txi) u_{k}+i\rho u_{k,\nu_k} +\cO(\gamma + \rho^2)u_k+ f_k,$$
we have
\begin{equation}\label{est-der-uks1}|\partial_{x_1}u_{k,\nu_k-1}|_{L^2}^2 \lesssim \rho^{2+2\epsilon}|u_k|^{2}_{L^2}
+\rho|<u_{k,\nu_k},\partial_{x_1}u_{k,\nu_k-1}>|+|f_k|_{L^1} |\partial_{x_1}u_{k,\nu_k}|_{L^\infty}.
\end{equation}
Here, integration by parts and Young's inequality yield
$$\begin{aligned}\rho|<u_{k,\nu_k},\partial_{x_1}u_{k,\nu_k-1}>|
&\lesssim \rho|\partial_{x_1}u_{k,\nu_k}|_{L^2}|u_{k,\nu_k-1}|_{L^2}
+ \rho|u_k(0)|^2.
\end{aligned}$$
Thus, using the refined bound \eqref{est-der-uks} and noting that
$$|u_k(0)|^2 \lesssim |<f,u_k>|\lesssim |f|_{L^1}|u_k|_{L^\infty}
\lesssim \rho^{-1}|f|_{L^1}^2,$$ we obtain
$$\begin{aligned}\rho|<u_{k,\nu_k},\partial_{x_1}u_{k,\nu_k-1}>| &\lesssim \rho^{1/2+\epsilon}\rho^{-3/2}(|f|_{L^1}^2+|f|_{L^\infty}^2)
\end{aligned}$$

Therefore, applying this estimate into \eqref{est-der-uks1}, we get \begin{equation}\label{est-der-uks-1}|\partial_{x_1}u_{k,\nu_k-1}|_{L^2}^2 \lesssim \rho^{-1+\epsilon}(|f|_{L^1}^2 + |f|_{L^\infty}^2).
\end{equation}
Using this refined bound, we can estimate the same for $u_{k,\nu_k-2}$, $u_{k,\nu_k-3}$, and so on. Thus, we obtain a refined bound for $u_k$: \begin{equation}\label{est--der-uk-refined}|\partial_{x_1}u_{k}|_{L^2}^2 \lesssim \rho^{-1+\epsilon}(|f|_{L^1}^2 + |f|_{L^\infty}^2)
\end{equation} where $\epsilon$ may be changed in each step and smaller than the original one. This and the standard Sobolev imbedding yield
\begin{equation}\label{est-uk-refined}|u_k|_{L^\infty}^2 \lesssim |u_k|_{L^2}|\partial_{x_1}u_{k}|_{L^2} \lesssim \rho^{-2+\epsilon}(|f|_{L^1}^2 + |f|_{L^\infty}^2)
\end{equation}
which proves the $L^\infty$ refined bound in \eqref{estZ-refined}
for $Z$. Using \eqref{est-uk-refined} into \eqref{Linftyest}, we
also obtain the refined bound in $L^2$ as claimed in
\eqref{estZ-refined}:
\begin{equation}|u_k|_{L^2}^2 \lesssim
\rho^{-3+\epsilon}(|f|_{L^1}^2 + |f|_{L^\infty}^2),
\end{equation} for some $\epsilon>0$.
\\
\\
{\bf Case b.} $\sigma \gtrsim \rho^\epsilon$ for some small
$\epsilon $ in $(0,1/2)$. We shall diagonalize this block. Recall
that
 \begin{equation}\label{gl-block1}
 Q_k ({\hat \zeta}, \rho)    =
i\underline {\mu}_k \Id  +i\left[\begin{array}{cccc} 0  & 1 & 0&
\\
0  &0  & \ddots  &  0
\\
  & \ddots &      \ddots & 1 \\
\sigma q_{\nu_k}&  &  0 & 0
\end{array}\right]+\cO(\sigma).
\end{equation}

Following \cite{Z2,Z3,GMWZ1}, we diagonalize this glancing block by
$$u'_{H_g}:=T^{-1}_{H_{g}} u_{H_g},$$ where $u_{H_g}:=u_{H_{g+}} + u_{H_{g-}}$. Here $u_{H_{g\pm}}$
are defined as the projections of $u_{H_g}$ onto the growing (resp.
decaying) eigenspaces of $Q_k ({\hat \zeta}, \rho)$ in
\eqref{gl-block1}. We recall the following whose proof can be found
in \cite{Z2,Z3} or Lemma 12.1, \cite{GMWZ1}.
\begin{lemma}[Lemma 12.1, \cite{GMWZ1}] The diagonalizing transformation $T_{H_g}$ may be chosen so that
\begin{equation}\label{T-est} |T_{H_g}|\le C, \qquad |T^{-1}_{H_g}|\le C\beta,\qquad |T^{-1}_{{H_g}|_{H_{g-}}}|\le C\alpha\end{equation}
where $\alpha,\beta$ are defined as
\begin{equation}\label{albeta}\beta := \sigma^{-1+1/{\nu_k}},\qquad\alpha
:= \sigma^{(1-[(\nu_k+1)/2])/\nu_k},\end{equation} and
$T^{-1}_{{H_g}|_{H_{g-}}}$ denotes the restriction of $T^{-1}_{H_g}$
to subspace $H_{g-}$. In particular, $\beta \alpha^{-2} \ge 1$.
\end{lemma}


Simple calculations show that eigenvalues of $Q_k$ are
\begin{equation} \alpha_{k,j} = i\ul{\mu_k}+\pi_{k,j} +
o(\sigma^{1/\nu_k}), \qquad j = 0,1,...,s-1.\end{equation} Here,
$\pi_{k,j}= \epsilon^j i (q_{\nu_k}\sigma)^{1/\nu_k}$, with
$\epsilon = 1^{1/\nu_k}$. We can further change of coordinates if
necessary to assume that
\begin{equation} Q'_k:=T^{-1}_{H_g}Q_k T_{H_g} =
 \diag (\alpha_{k,1},\cdots, \alpha_{k,l},\alpha_{k,l+1},\cdots ,\alpha_{k,\nu_k})\end{equation} with
 \begin{equation}\label{rea-est}\begin{aligned}-&\R ~\alpha_{k,j} >0, \quad j = 1,...,l,\\
 &\R ~\alpha_{k,j} >0 , \quad j = l+1,...,\nu_k.\end{aligned}\end{equation}

Hence, applying Lemma \ref{lem-bounds} to equations of $u'_{H_g}$
with $S = Id$ or $S = -Id$, we easily obtain
\begin{equation}\label{glbound}\begin{aligned} |u'_{H_{g+}}|^2_{L^\infty}+ \rho
\min_j|\R~\alpha_{k,j}||u'_{H_{g+}}|_{L^2}^2 &\lesssim
|F'_{H_{g+}}|_{L^1}^2,
\\|u'_{H_{g-}}|^2_{L^\infty} + \rho\min_j|\R~\alpha_{k,j}||u'_{H_{g-}}|_{L^2}^2&\lesssim |u'_{H_{g-}}(0)|^2
+ |F'_{H_{g-}}|_{L^1}^2.\end{aligned}\end{equation} The diagonalized
boundary condition $\Gamma':=\Gamma_a T_{H_g}$. By computing, we
observe that
$$ |\Gamma' u'_{H_{g-}}| =
|\Gamma u_{H_{g-}}| \ge C^{-1} |u_{H_{g-}}| \ge
\frac{C^{-1}|u'_{H_{g-}}|}{|T^{-1}_{{H_g}|_{H_{g-}}}|} \ge
C^{-1}\alpha^{-1}||u'_{H_{g-}}|.$$ Thus,
\begin{equation}\begin{aligned} |u'_{H_{g-}}| &\le C\alpha  |\Gamma' u'_{H_{g-}}|
\le C\alpha(|\Gamma' u'|+|\Gamma'u'_+|) \le
C\alpha|u'_+|.\end{aligned}\end{equation} Using this estimate,
\eqref{T-est}, and \eqref{goodbound1}, the estimate \eqref{glbound}
yields
\begin{equation}\label{glbound1}\begin{aligned}
\alpha^{-2}|u_{H_{g}}|^2_{L^\infty}+ \rho
\alpha^{-2}\min_j|\R~\alpha_{k,j}||u_{H_{g}}|_{L^2}^2 &\lesssim
\beta^2|f|_{L^1}^2.\end{aligned}\end{equation}

Recalling that $\alpha,\beta$ are defined as in \eqref{albeta} and
the fact that we are in the case of $\sigma \ge \rho^\epsilon$ for
some small $\epsilon>0$, we get
\begin{equation}\label{goodbound2} |u_{H_{g}}|_{L^\infty} \le C
\alpha \beta |f|_{L^1} \le C \rho^{-2\epsilon}
|f|_{L^1},\end{equation} from which we obtain the refined bounds
\eqref{estZ-refined} for this case as well.

\begin{remark}\label{rem-H4'}\textup{In case b) above,
we use the nonvanishing of $q_{\nu_k}$ to make sure that $\sigma
q_{\nu_k}$ is order of $\sigma$ in the neighborhood $\omega$ of
$(\underline{\hat\zeta},0)$ so that the lower left hand entry of
$Q_k$ dominates and thus we can be sure to diagonalize the block.
Otherwise, the other entries of $Q_k$ in \eqref{gl-block1} may
dominate and the behavior is not clear. The nonvanishing of
$q_{\nu_k}$ is guaranteed by our additional Hypothesis (H4$'$) as
shown in the proof of Proposition \ref{prop-blockst}. This is only
place  in the paper where the assumption (H4$'$) is used.}
\end{remark}

\subsection{$L^1\to L^p$ estimates} We establish the $L^1\to L^p$ resolvent bounds for low frequency regime,
restricting our attention to the surface
\begin{equation}\label{gamma}\Gamma^{\tilde \xi}:=\{\lambda~:~\R\lambda =
-\theta_1(|\tilde \xi|^2 + |\I\lambda|^2)\},\end{equation} for
$\theta_1>0$. Taking $\theta_1$ to be sufficiently small such that
all earlier resolvent estimates are still valid on $\Gamma^\txi$,
with $\rho :=|(\tilde \xi,\lambda)|$ being sufficiently small. Thus,
we obtain the following:

\begin{proposition}[Low-frequency bounds]\label{prop-resLF} Under the hypotheses of Theorem \ref{theo-nonlin}, for $\lambda \in \Gamma^{\tilde \xi}$ and
$\rho :=|(\tilde
\xi,\lambda)|$, $\theta_1$ sufficiently small, there holds the
resolvent bound \begin{equation}\label{res-bound} |(L_{\tilde
\xi}-\lambda)^{-1}\partial_{x_1}^\beta f|_{L^p(x_1)} \le
C\rho^{-1-1/p+\epsilon}[\rho^\beta| f|_{L^1(x_1)}+
|f|_{L^\infty(x_1)}],\end{equation} for all $2\le p\le \infty$,
$\beta =0,1$, and $\epsilon>0$.
\end{proposition}

\begin{proof} Following \cite{Z2,Z3}, define the curves
$$(\txi,\lambda)(\rho,\hat\txi,\hat \tau):=(\rho \hat \txi,\rho i\hat \tau-
\theta_1\rho^2),$$ where $\hat\txi \in \RR^{d-1}, \hat\tau \in \RR$
and $(\hat\txi,\hat \tau)\in S^d$: $|\hat\txi|^2 +|\hat \tau|^2 =1$.
As $(\rho,\hat\txi,\hat \tau)$ range in the compact set
$[0,\delta]\times S^d$, $(\txi,\lambda)$ traces out the portion of
the surface $\Gamma^\txi$ contained in the set $|\txi|^2 +
|\lambda|^2 \le \delta$. Thus, using $L^2$ and $L^\infty$ estimates
obtained in previous sections with $\hat \gamma=0$ and applying the
interpolation inequality between $L^2$ and $L^\infty$ spaces, we
obtain the proposition in the case $\beta=0$.


Now, recalling that $W = \Phi V Z$ and all coordinate transformation
matrices are uniformly bounded, the refined bounds of $Z$ therefore
imply improved bounds for $W$ and thus $U$. Bounds for $L^p$,
$2<p<\infty$, are obtained by interpolation inequality between $L^2$
and $L^\infty$. Hence, we have proved the bounds for $\beta =0$ as
claimed.


For $\beta =1$, we expect that $\partial_{x_1} f$ plays a role as
``$\rho f$'' forcing. Recall that the eigenvalue equations $(L_\txi
- \lambda)U =\partial_{x_1}f$ read
\begin{equation}\label{eg-eqs}\begin{aligned}\overbrace{(B^{11}U_{x_1})_{x_1}-(A^1U)_{x_1}}^{L_0U} - &i\sum_{j\not=1}A^j\xi_jU +
i\sum_{j\not=1}B^{j1}\xi_jU_{x_1}
\\&+i\sum_{k\not=1}(B^{1k}\xi_kU)_{x_1} -
\sum_{j,k\not=1}B^{jk}\xi_j\xi_k U - \lambda U =\partial_{x_1}f.
\end{aligned}\end{equation}

Now modifying the nice argument of Kreiss-Kreiss presented in
\cite{KK,GMWZ1}, we write $U = V + U_1$, where $V$ satisfies
\begin{equation}\label{auxeqs} (L_0-\lambda) V =
\partial_{x_1} f,\qquad x_1\in \RR. \end{equation}Noting that $A^1$ and
$B^{11}$ depend on $x_1$ only, we thus obtain by one-dimensional
results (see \cite{MaZ3,Z3}) the following pointwise bounds on Green
kernel $G^0_\lambda$ of $\lambda - L_0$,
\begin{equation} \label{ptbounds} |\partial_{y_1}G^0_\lambda(x_1,y_1)|
\le C e^{-\rho|x_1-y_1|}(\rho + e^{-\theta|y_1|}).\end{equation}

Hence, employing Hausdorff-Young's inequality, we obtain
\begin{equation}\label{Vbound} |V|_{L^p(x_1)}+ |V_{x_1}|_{L^p(x_1)} \le
C\rho^{-1/p}[\rho|f|_{L^1(x_1)} + |f|_{L^\infty(x_1)}],\end{equation} for
all $1\le p\le \infty$.

Now from $U_1 = U-V$ and equations of $U$ and $V$, we observe that
$U_1$ satisfies
\begin{equation}\label{U1eqs}\begin{aligned}(L_\txi  - \lambda )U_1 = L(V,V_{x_1}),
\end{aligned}\end{equation} where $ L(V,V_{x_1}) = \rho\cO(|V|+|V_{x_1}|)$.

Therefore applying the result which we just proved for $\beta=0$ to
the equations \eqref{U1eqs}, we obtain
\begin{equation}\begin{aligned}|U_1|_{L^p(x_1)} &\le
C\rho^{-1-1/p+\epsilon}\Big[|L(V,V_{x_1})|_{L^1(x_1)}+|L(V,V_{x_1})|_{L^\infty(x_1)}\Big]\\&\le C\rho^{-1-1/p+\epsilon}\rho
\Big[|V|_{L^q}+|V_{x_1}|_{L^q}\Big] \\&\le
C\rho^{-1/p+\epsilon}[|f|_{L^1(x_1)}+
\rho^{-1}|f|_{L^\infty(x_1)}].\end{aligned}\end{equation}

Bounds on $V$ and $U_1$ clearly give our claimed bounds on $U$ by
triangle inequality: $$|U|_{L^p} \le |V|_{L^p}+|U_1|_{L^p}.$$ We
obtain the proposition for the case $\beta=1$, and thus complete the
proof.

\end{proof}

\subsection{Estimates on the solution operator}\label{sec-estS-comp} In this subsection, we complete the proof of Proposition \ref{prop-estS}. As mentioned earlier, it suffices to prove the bounds for $\cS_1(t)$, where the low frequency solution
operator $\cS_1(t)$ is defined as
\begin{equation}\label{cS1}\cS_1(t):=\frac{1}{(2\pi i)^d}\int_{|\tilde \xi|\le
r}\oint_{\Gamma^{\tilde \xi}} e^{\lambda t + i\tilde \xi \cdot\tilde
x}(L_{\txi} - \lambda)^{-1} d\lambda d\txi.\end{equation}

\begin{proof}[Proof of bounds on $\cS_1(t)$] Let $\hat u(x_1,\txi,\lambda)$ denote the solution of
$(L_\txi-\lambda)\hat u = \hat f$, where $\hat f(x_1,\txi)$ denotes
Fourier transform of $f$, and
$$u(x,t):=\cS_1(t)f = \frac{1}{(2\pi i)^d}\int_{|\txi|\le r}\oint _{\Gamma^\txi}
e^{\lambda t+i\txi \cdot \tx}(L_\txi - \lambda)^{-1}\hat
f(x_1,\txi)d\lambda d\txi.$$

Using Parseval's identity, Fubini's theorem, the triangle
inequality, and Proposition \ref{prop-resLF}, we may estimate $$\begin{aligned}
|u|_{L^2(x_1,\tx)}^2(t) &= \frac{1}{(2\pi)^{2d}}\int_{x_1}
\int_{\txi}\Big|\oint_{\Gamma^\txi} e^{\lambda t}\hat
u(x_1,\txi,\lambda)d\lambda\Big|^2 d\txi dx_1
\\&\le
\frac{1}{(2\pi)^{2d}}\int_{\txi}\Big|\oint_{\Gamma^\txi}
e^{\R\lambda t}|\hat u(x_1,\txi,\lambda)|_{L^2(x_1)}d\lambda\Big|^2
d\txi \\&\le C[|f|_{L^1(x)}+|f|_{L^{1,\infty}_{\tx,x_1}}]^2\int_{\txi}\Big|\oint_{\Gamma^\txi} e^{\R\lambda
t}\rho^{-3/2+\epsilon}d\lambda\Big|^2 d\txi.
\end{aligned}$$

Specifically, parametrizing $\Gamma^\txi$ by $$\lambda(\txi,k) = ik
- \theta_1(k^2 + |\txi|^2), \quad k\in \RR,$$ we estimate
$$\begin{aligned}
\int_{\txi}\Big|\oint_{\Gamma^\txi} e^{\R\lambda
t}\rho^{-3/2+\epsilon}d\lambda\Big|^2 d\txi &\le
\int_{\txi}\Big|\int_\RR e^{-\theta_1(k^2+|\txi|^2)
t}\rho^{-3/2+\epsilon}dk\Big|^2 d\txi\\&\le
\int_{\txi}e^{-2\theta_1|\txi|^2t}|\txi|^{-1}\Big|\int_\RR
e^{-\theta_1k^2t}|k|^{\epsilon-1}dk\Big|^2 d\txi
\\&\le
Ct^{-(d-2)/2-\epsilon},
\end{aligned}$$ noting that $\int_{\RR^{d-1}} e^{-\theta |x|^2} |x|^{-\alpha} dx$ is finite, provided $\alpha < d-1$.

Similarly, we estimate $$\begin{aligned}
|u|_{L^{2,\infty}_{\tx,x_1}}^2(t)&\le
\frac{1}{(2\pi)^{2d}}\int_{\txi}\Big|\oint_{\Gamma^\txi}
e^{\R\lambda t}|\hat
u(x_1,\txi,\lambda)|_{L^\infty(x_1)}d\lambda\Big|^2 d\txi \\&\le
C[|f|_{L^1(x)}+|f|_{L^{1,\infty}_{\tx,x_1}}]^2\int_{\txi}\Big|\oint_{\Gamma^\txi} e^{\R\lambda
t}\rho^{-1+\epsilon}d\lambda\Big|^2 d\txi
\end{aligned}$$
where, parametrizing $\Gamma^\txi$ as above, we have
$$\begin{aligned}
\int_{\txi}\Big|\oint_{\Gamma^\txi} e^{\R\lambda
t}\rho^{-1+\epsilon}d\lambda\Big|^2 d\txi &\le
\int_{\txi}e^{-\theta_1|\txi|^2 t}\Big|\int_\RR
e^{-\theta_1k^2t}|k|^{\epsilon-1}dk\Big|^2d\txi\\&\le Ct^{-(d-1)/2
-\epsilon}.
\end{aligned}$$

Finally, we estimate $$\begin{aligned} |u|_{L^\infty_{\tx, x_1}}(t)
&\le\frac{1}{(2\pi)^{d}} \int_{\txi}\oint_{\Gamma^\txi} e^{\R\lambda
t}|\hat u(x_1,\txi,\lambda)|_{L^\infty(x_1)}d\lambda d\txi \\&\le
C[|f|_{L^1(x)}+|f|_{L^{1,\infty}_{\tx,x_1}}]\int_{\txi}\oint_{\Gamma^\txi} e^{\R\lambda
t}\rho^{-1+\epsilon}d\lambda d\txi
\end{aligned}$$
where, parametrizing $\Gamma^\txi$ as above, we have
$$\begin{aligned}
\int_{\txi}\oint_{\Gamma^\txi} e^{\R\lambda t}\rho^{-1+\epsilon}d\lambda
d\txi &\le \int_{\txi}e^{-\theta_1|\txi|^2
t}\int_\RR
e^{-\theta_1k^2t}|k|^{\epsilon-1}dkd\txi\\&\le Ct^{-(d-1)/2-\epsilon/2}.
\end{aligned}$$ The $x_1-$derivative bounds follow similarly by using the version of the $L^1\to L^p$ estimates for $\beta_1=1$. The $\tx-$derivative bounds are straightforward by the fact that $\widehat{\partial_{\tx}^{\tilde \beta} f} = (i\txi)^{\tilde \beta} \hat f$.
\end{proof}

\section{Two--dimensional case or cases with (H4)}\label{sec-2dcase}
In this section, we give 
an immediate proof of Theorem \ref{theo-stabH4}. Notice that the
only assumption we make here that differs from those in \cite{NZ2}
is the relaxed Hypothesis (H3$'$), treating the case of totally
nonglancing characteristic
roots, which is only involved in low--frequency estimates. That is to say, we 
only need to establish the $L^1\to L^p$ bounds in low-frequency
regimes for this new case.

\begin{proposition}[Low-frequency bounds; \cite{NZ2}, Proposition 3.3]\label{prop-resLFH5} Under
 the hypotheses of Theorem \ref{theo-stabH4},
for $\lambda \in \Gamma^{\tilde \xi}$ (see \eqref{gamma}) and $\rho
:=|(\tilde \xi,\lambda)|$, $\theta_1$ sufficiently small, there
holds the resolvent bound \begin{equation}\label{res-boundH5}
|(L_{\tilde \xi}-\lambda)^{-1}\partial_{x_1}^\beta f|_{L^p(x_1)} \le
C\gamma_2\rho^{-2/p}\Big[\rho^{\beta}|\hat f|_{L^1(x_1)} +
\beta|\hat f|_{L^\infty(x_1)} \Big],\end{equation} for all $2\le
p\le \infty$, $\beta =0,1$, and $\gamma_2$ is the diagonalization
error (see \cite{Z3}, (5.40)) defined as
\begin{equation}\label{gamma2} \gamma_2 := 1+ \sum_{j,\pm}\Big[\rho^{-1}|\I\lambda
 - \eta_j^\pm(\txi)|+\rho\Big]^{1/s_j-1},\end{equation} with
 $\eta_j^\pm,s_j$ as in (H4).
\end{proposition}

\begin{proof} We only need to treat the new case: the totally nonglancing blocks
$Q_t^k$. But this is already treated in our previous subsection,
Subsection \ref{subsec-totalnongl}, yielding
\begin{equation}\label{totnongl}\begin{aligned} |u_{H_{t+}}|^2_{L^\infty}+ \rho^2|u_{H_{t+}}|_{L^2}^2 &\lesssim |F_{H_{t+}}|_{L^1}^2,
\\|u_{H_{t-}}|^2_{L^\infty} + \rho^2|u_{H_{t-}}|_{L^2}^2&\lesssim |u_{H_{t-}}(0)|^2
+ |F_{H_{t-}}|_{L^1}^2,\end{aligned}\end{equation}where the boundary
term $|u_{H_{t-}}(0)|^2$ can be treated by applying the $L^2$
stability estimate \eqref{max-est}. Thus, together with a use of the
standard interpolation inequality, we have obtained
\begin{equation}\label{nonglest}
|u_{H_t}|_{L^p(x_1)} \le C\gamma_2\rho^{-1}|
f|_{L^1(x_1)},\end{equation} for all $2\le p\le \infty$ and
$\gamma_2$ defined as in \eqref{gamma2}, yielding
\eqref{res-boundH5} for $\beta=0$. For $\beta=1$, we can follow the
Kreiss--Kreiss trick as done in the proof of Proposition
\ref{prop-resLF}, completing the proof of Proposition
\ref{prop-resLFH5}. \end{proof}

\begin{proof}[Proof of Theorem \ref{theo-stabH4}] Proposition \ref{prop-resLFH5} is Proposition 3.3 in
\cite{NZ2} with an extension to the totally nonglancing cases. Thus,
we can now follow word by word the proof in \cite{NZ2}, yielding the
theorem.
\end{proof}


\appendix

\section{Genericity of (H4$'$)}\label{genH4}

Genericity of our additional structural assumption (H4$'$) is clear.
Indeed, violation of the condition would require $d$ equations:
$\partial_{\xi_j}\lambda_k(\xi) =0$ for all $j = 1,\cdots,d$,
whereas only $d-1$ parameters in $\xi\in \RR^d\setminus \{0\}$ are
varied as $\xi$ may be constrained in the unit sphere $S^d$ by
homogeneity of $\lambda(\xi)$ in $\xi$.

Finally, we give the following counterexample of Kevin Zumbrun in
the two--dimensional case for which the hypothesis (H4$'$) fails.

\begin{counterexample}\textup{Let
\begin{equation}A_1:=\begin{pmatrix} 0&1\\1&0
\end{pmatrix}\qquad A_2:=\begin{pmatrix} 0&0\\0&1
\end{pmatrix}.\end{equation}
Then both $A_1$ and $A_2$ are clearly symmetric and do not commute.
However, at $\xi_1=0$, the matrix $\xi_1A_1 + \xi_2 A_2$ has an
eigenvalue ($\lambda(\xi) \equiv 0$) such that $\nabla\lambda=0$,
violating (H4$'$). }
\end{counterexample}

Counterexamples for higher--dimensional cases can be constructed
similarly.


\begin{thebibliography}{GMWZ6}




%
%



\bibitem[CHNZ]{CHNZ}
N.~Costanzino, J.~Humpherys, T.~Nguyen, and K.~Zumbrun, {\it
Spectral stability of noncharacteristic boundary layers of
isentropic Navier--Stokes equations,} Preprint, 2007.






\bibitem[GG]{GG}
Grenier, E. and Gu\`es, O., \emph{Boundary layers for viscous
perturbations of noncharacteristic quasilinear hyperbolic problems},
J. Differential Eqns. 143 (1998), 110-146.

\bibitem[GR]{GR}
Grenier, E. and Rousset, F., \emph{Stability of one dimensional
boundary layers by using Green's functions}, Comm. Pure Appl. Math.
54 (2001), 1343-1385.

\bibitem[GMWZ1]{GMWZ1}
O. Gu{\`e}s, G.~M{\'e}tivier, M.~Williams, and K.~Zumbrun. {\it
Multidimensional viscous shocks I: degenerate symmetrizers and long
time stability,} J. Amer. Math. Soc.  18  (2005),  no. 1, 61--120.

\bibitem[GMWZ5]{GMWZ5}
O. Gu{\`e}s, G.~M{\'e}tivier, M.~Williams, and K.~Zumbrun.
\newblock {\em Existence and stability of noncharacteristic
hyperbolic-parabolic boundary-layers.}
\newblock Preprint, 2008.

\bibitem[GMWZ6]{GMWZ6}
O. Gu{\`e}s, G.~M{\'e}tivier, M.~Williams, and K.~Zumbrun.
\newblock {\em Viscous boundary value problems for symmetric systems with variable
multiplicities},
\newblock  J. Differential Equations 244 (2008) 309–387.





\bibitem[HLZ]{HLZ}
J.~Humpherys, O.~Lafitte, and K.~Zumbrun. {\it Stability of viscous
shock profiles in the high Mach number limit,} (Preprint, 2007).

\bibitem[HLyZ1]{HLyZ1} Humpherys, J., Lyng, G., and Zumbrun, K.,
\emph{Spectral stability of ideal-gas shock layers}, Preprint
(2007).

\bibitem[HLyZ2]{HLyZ2} Humpherys, J., Lyng, G., and Zumbrun, K.,
\emph{Multidimensional spectral stability of large-amplitude
Navier-Stokes shocks}, in preparation.
%

\bibitem[KaK]{KaK} Y. Kagei and S. Kawashima \emph{Stability of planar stationary solutions to the compressible
Navier-Stokes equations in the half space,}
Comm. Math. Phys. 266 (2006), 401-430.




\bibitem[KNZ]{KNZ} S. Kawashima, S. Nishibata, and P. Zhu,
\emph{Asymptotic stability of the stationary solution to the
compressible
  Navier-Stokes equations in the half space,}
Comm. Math. Phys. 240 (2003), no. 3, 483--500.
%
%
%
%
%
%

\bibitem[KK]{KK} Kreiss, G. and Kreiss, H.-O.,
\newblock {\em Stability of systems of viscous conservation laws,}
\newblock {Comm. Pure Appl. Math.}, 50, 1998, 1397--1424.








\bibitem[MaZ3]{MaZ3}
C.~Mascia and K.~Zumbrun.
\newblock {\em Pointwise {G}reen function bounds for shock profiles of systems with
  real viscosity.}
\newblock {Arch. Ration. Mech. Anal.}, 169(3):177--263, 2003.

\bibitem[MaZ4]{MaZ4}
C.~Mascia and K.~Zumbrun.
\newblock {\em Stability of large-amplitude viscous shock profiles of
  hyperbolic-parabolic systems.}
\newblock {Arch. Ration. Mech. Anal.}, 172(1):93--131, 2004.

\bibitem[MN]{MN}  Matsumura, A. and Nishihara, K.,
{\it Large-time behaviors of solutions to an inflow problem in
the half space for a one-dimensional system of compressible viscous gas,}
Comm. Math. Phys., 222 (2001), no. 3, 449--474.


\bibitem[Met]{Met} G. M\'etivier,
\emph{Small Viscosity and Boundary Layer Methods,}
 Birkh\"{a}user, Boston
 2004.


\bibitem[MeZ1]{MeZ1}
M\'etivier, G. and Zumbrun, K., \textit{Large viscous boundary
layers for noncharacteristic nonlinear hyperbolic problems}, Memoirs
AMS, 826 (2005).


\bibitem[MeZ3]{MeZ3}
M\'etivier, G. and Zumbrun, K., \textit{Hyperbolic boundary value
problems for symmetric systems with variable multiplicities}, J.
Diff. Eqns., 211, (2005), 61--134.



\bibitem[N1]{N1}
T. Nguyen, {\it Stability of multi-dimensional viscous shocks for
symmetric systems with variable multiplicities}, Preprint, 2008




\bibitem[NZ1]{NZ1}
T. Nguyen and K. Zumbrun, {\it Long-time stability of
large-amplitude noncharacteristic boundary layers for
hyperbolic-parabolic systems}, Preprint, 2008


\bibitem[NZ2]{NZ2}
T. Nguyen and K. Zumbrun, {\it Long-time stability of multi-dimensional noncharacteristic
viscous boundary layers}, Preprint, 2008





%
%





\bibitem[R3]{R3} Rousset, F.,
{\it  Stability of small amplitude boundary layers for mixed
hyperbolic-parabolic systems,} Trans. Amer. Math. Soc.  355  (2003),
no. 7, 2991--3008.

%
\bibitem[S]{S} H. Schlichting,
{\em Boundary layer theory}, Translated by J. Kestin. 4th ed.
McGraw-Hill Series in Mechanical
 Engineering. McGraw-Hill Book Co., Inc., New York, 1960.

\bibitem[SGKO]{SGKO} H. Schlichting, K. Gersten, E. Krause, and H. Jr.
Oertel, {\em Boundary-Layer Theory },  Springer; 8th ed. 2000. Corr.
2nd printing edition (March 22, 2004)

\bibitem[SZ]{SZ}
Serre, D. and Zumbrun, K., {\it Boundary layer stability in real
vanishing-viscosity limit}, Comm. Math. Phys. 221 (2001), no. 2,
267--292.



\bibitem[YZ]{YZ} S. Yarahmadian and K. Zumbrun,
{\it Pointwise Green function bounds and long-time stability of
large-amplitude noncharacteristic boundary layers}, Preprint (2008).


\bibitem[Z2]{Z2}
K.~Zumbrun.
\newblock Multidimensional stability of planar viscous shock waves.
\newblock In {\em Advances in the theory of shock waves}, volume~47 of {\em
  Progr. Nonlinear Differential Equations Appl.}, pages 307--516. Birkh\"auser
  Boston, Boston, MA, 2001.

\bibitem[Z3]{Z3}
K.~Zumbrun.
\newblock Stability of large-amplitude shock waves of compressible
  {N}avier-{S}tokes equations.
\newblock In {\em Handbook of mathematical fluid dynamics. Vol. III}, pages
  311--533. North-Holland, Amsterdam, 2004.
\newblock With an appendix by Helge Kristian Jenssen and Gregory Lyng.

\bibitem[Z4]{Z4}
K.~Zumbrun.
\newblock Planar stability criteria for viscous shock waves of systems with
  real viscosity.
\newblock In {\em Hyperbolic systems of balance laws}, volume 1911 of {\em
  Lecture Notes in Math.}, pages 229--326. Springer, Berlin, 2007.

\bibitem[Z5]{Z5}
K.~Zumbrun, {\it Stability of noncharacteristic boundary layers in
the standing shock limit}, Preprint, 2008.





\end{thebibliography}
\end{document}